\documentclass[reqno,12pt]{article}
\usepackage{amsmath,amsthm,amssymb,bm}
\usepackage{datetime}

\newtheorem{theorem}{Theorem}

\newtheorem{lemma}[theorem]{Lemma}
\newtheorem{corollary}[theorem]{Corollary}

\usepackage{tikz}
\usetikzlibrary{shapes.geometric}
\usetikzlibrary{fit}
\usetikzlibrary{calc}

\long\def\symbolfootnote[#1]#2{\begingroup\def\thefootnote{\fnsymbol{footnote}}
\footnote[#1]{#2}\endgroup}

\begin{document}

\title{On lower bounds for the matching number of subcubic graphs}
\author{P.E.~Haxell\footnote{Department of Combinatorics and
Optimization, University of Waterloo, Waterloo, Ontario, Canada N2L
3G1. pehaxell@math.uwaterloo.ca; Partially
supported by NSERC.}\ \ 
and A.D.~Scott\footnote{Mathematical Institute,
    University of Oxford, Andrew Wiles Building, Radcliffe Observatory
    Quarter, Woodstock Rd, Oxford, UK OX2 6GG. scott@maths.ox.ac.uk}}
\date{\today}

\maketitle

\begin{abstract}
We give a complete description of the set of triples
$(\alpha,\beta,\gamma)$ of real numbers with the following property. There
exists a constant $K$ such that $\alpha n_3+\beta n_2+\gamma n_1-K$ is
a lower bound for the matching number 
$\nu(G)$ of every connected subcubic graph $G$, where $n_i$ denotes
the number of vertices of degree $i$ for each $i$.
\end{abstract}

\noindent {\bf Keywords:} matching, subcubic graph, polyhedron


\section{Introduction}

A graph is said to be {\it subcubic} if its maximum degree is at most
three. In this paper we consider lower bounds for the maximum size
$\nu(G)$ of a matching in subcubic graphs $G$.

Various lower bounds on $\nu(G)$ for subcubic graphs $G$ appear in the
literature. For example, the following theorem is due to Biedl,
Demaine, Duncan, Fleischer and Kobourov~\cite{BD}. Here
$n_i$ denotes the number of vertices of degree $i$ in $G$,
and $\ell_2$ denotes the number of
end-blocks in the block-cutvertex tree of $G$.

\begin{theorem} \label{bdetc}
Let $G$ be a connected graph with $n$ vertices.
\begin{enumerate}
\item If $G$ is cubic then $\nu(G)\geq 4(n-1)/9$.
\item If $G$ is subcubic then $\nu(G)\geq n_3/2+n_2/3+n_1/2-\ell_2/3$, and $\nu(G)\ge (n-1)/3$. 
\end{enumerate}
\end{theorem}

They also asked whether $\nu(G)\ge (3n+n_2)/9$ for every subcubic graph.  
It will turn out below that this is not the case.

Generalisations of~\cite{BD} to regular graphs of higher degree
were given by Henning and Yeo in~\cite{HY} (see also O and
West~\cite{OW}). Lower bounds in terms of other parameters of $G$ have
been given, for example, in~\cite{OW} and~\cite{HLR}.

Our aim in this paper is to give a {\em complete} description of the set $L$ of 3-tuples of real coefficients
$(\alpha,\beta,\gamma)$ for which there exists a constant $K$ such that $\nu(G)\geq\alpha
n_3+\beta n_2+\gamma n_1-K$ for every connected subcubic graph
$G$. (Note that this is equivalent to saying $\nu(G)\geq\alpha         
n_3+\beta n_2+\gamma n_1-Kc(G)$ for every subcubic graph
$G$, where $c(G)$ denotes the number of components of $G$.) Our work
here is similar in spirit to a 
result of Chv\'atal and McDiarmid~\cite{CM}, who addressed a similar question for
cover numbers of hypergraphs in terms of their number of vertices and
number of edges. We will find, as in~\cite{CM}, that $L$ is a convex
set, but in contrast to~\cite{CM} where the number of extreme points is
infinite, in our case $L$ is a certain 3-dimensional polyhedron with a
relatively simple description.

We define the polyhedron $P\subset\mathbb{R}^3$ to be the intersection of the six
half-spaces 
\begin{align*}
 x_3&\leq4/9,\\
 x_2&\leq1/2,\\
 x_3+x_1&\leq2/3,\\
 x_3+3x_2/2&\leq1,\\
 x_3+x_2+x_1&\leq1,\\ 
 x_3+x_2/6&\leq1/2.
\end{align*}
We let $P_+$ be the intersection of $P$ with the nonnegative orthant $[0,\infty)^3$ in ${\mathbb{R}}^3$.
It is easily seen that $P$ is unbounded.  However,
it follows from the first three inequalities above that $P_+$  is a bounded subset of the nonnegative orthant.

The main aim of this paper is to prove the following theorem.

\begin{theorem}\label{pl}
$P=L$.
\end{theorem}

We will prove that $P\subseteq L$ in Section \ref{pmain}, and $L\subseteq P$ in Section \ref{pmain2}.  

Our proof that $P\subseteq L$ will need the fact that five specific points belong to $L$.  This
is a consequence of the following stronger result, which we prove in Section \ref{main}.

\begin{theorem}\label{main} Let $G$ be a subcubic graph with $c=c(G)$
  components. Then
\begin{align}
\nu(G)&\geq n_2/2+n_1/2-c/2,\label{b1}\\
\nu(G)&\geq n_2/3+2n_1/3-c,\label{b2}\\
\nu(G)&\geq n_3/4+n_2/2+n_1/4-c/2,\label{b3}\\
\nu(G)&\geq 7n_3/16+3n_2/8+3n_1/16-c/8,\label{b4}\\
\nu(G)&\geq 4n_3/9+n_2/3+2n_1/9-c/9.\label{b5}
\end{align} 
\end{theorem} 

All five of these bounds are sharp: \eqref{b4} is attained
by the triangle, \eqref{b1} and \eqref{b3} by any odd cycle, and \eqref{b1}, \eqref{b2} and \eqref{b5} by
the claw $K_{1,3}$.  Furthermore, for a subcubic graph $G$, each of the bounds is sharp for $G$ if and only if it is
sharp for every component of $G$.
We will give further connected, sharp examples for \eqref{b1}, \eqref{b2}, \eqref{b3}, \eqref{b5} in Section \ref{pmain2}.
The proof of Theorem~\ref{main} is given in
Section~\ref{pmainb}, where
we will also note the following corollary concerning the constant $K$ from the definition of $L$.

\begin{corollary}\label{pmcor}
Let $(\alpha,\beta,\gamma)$ be an element of $P$.
\begin{enumerate}
\item If $\alpha\geq 0$ then  $\nu(G)\ge\alpha n_3+\beta n_2+\gamma n_1-1$ 
for every connected subcubic graph $G$.
\item If $\alpha<0$ then $\nu(G)\ge\alpha n_3+\beta n_2+\gamma n_1-(2|\alpha|+1)$
for every connected subcubic graph $G$.
\end{enumerate}
 \end{corollary}

Note in particular that if $G$ is a connected subcubic graph then
$\nu(G)\ge\alpha n_3+\beta n_2+\gamma n_1-1$ for every
$(\alpha,\beta,\gamma)\in P_+$. Note also that if we consider
$G=K_{1,3}$ and $(\alpha,\beta,\gamma)=(-\lambda,0,\lambda+2/3)$
(which is in $P$ for all $\lambda\geq 0$), then the first bound in
Lemma~\ref{pmcor} is sharp for $\lambda=0$, and the second is sharp
for all $\lambda>0$.

In the other direction, the fact that $L\subseteq P$ is a consequence of the following result, which we will prove in 
Section \ref{pmain2}.

\begin{theorem}\label{main2}
If $(\alpha,\beta,\gamma)\not\in P$ then for every
constant $K$ there exists a connected subcubic graph $G$ such that
$\nu(G)<\alpha n_3+\beta n_2+\gamma n_1-K$.
\end{theorem}

Our results generalize previous work.
For example, the first bound in Theorem \ref{bdetc} is a special case of \eqref{b5}; the bound $\nu\ge(n-1)/3$ follows from
a convex combination of \eqref{b2} and \eqref{b5}.  On the other hand, the answer to the question of
Biedl, Demaine, Duncan, Fleischer and Kobourov~\cite{BD} as to 
whether $\nu(G)\ge (3n+n_2)/9$ for every subcubic graph is negative
by Theorem \ref{pl}: the vector $(1/3,4/9,1/3)$ is not in $P$ as it violates the 
inequality $x_1+x_2+x_3\le1$, and Example 3 in Section \ref{pmain2} is a counterexample.

\section{$P\subseteq L$}\label{pmain}

In this section we prove one direction of Theorem \ref{pl}, namely that $P\subseteq L$ 
(leaving aside the proof of Theorem \ref{main}, which we defer to the next section).
We will prove that $P\subseteq L$ in two steps.  We first show that it is enough to consider just $P_+$, and then prove that $P_+\subseteq L$.

We begin with the following simple but useful observation.

\begin{lemma}\label{n1n3}
In any connected subcubic graph $G$ we have $n_3\geq n_1-2$.
\end{lemma}
\begin{proof}
Let $T$ be a spanning tree of $G$, and
let $t_i$ denote the number of vertices of degree $i$ in $T$. Then
$t_1\geq n_1$, $t_3\leq n_3$, and $t_1=t_3+2$. Thus $n_3\geq n_1-2$.
\end{proof}

Next we note 
some closure properties of $L$.

\begin{lemma}\label{lamb}
\begin{enumerate}
\item $L$ is convex.

\item $L$ is downward closed: if $(a_3,a_2,a_1)\in L$ and $b_i\le a_i$ for all $i$ then $(b_3,b_2,b_1)\in L$.  

\item If $(x_3,x_2,x_1)\in L$ then $(x_3-\lambda,x_2,x_1+\lambda)\in L$ for
all $\lambda\geq 0$.
\end{enumerate} 
\end{lemma}

\begin{proof}
Suppose that $\mathbf a=(a_3,a_2,a_1)$, $\mathbf b=(b_3,b_2,b_1)$ lie in $L$, with associated constants $K_a, K_b$.  Thus for every subcubic graph $G$, 
say with parameters  $\mathbf n=(n_3,n_2,n_1)$ and matching number $\nu$, we have $\mathbf a\cdot\mathbf n \leq \nu+K_a$ and  $\mathbf b\cdot\mathbf n \leq \nu+K_b$.
Suppose that $\lambda\in[0,1]$ and $\mathbf c=\lambda\mathbf a+(1-\lambda)\mathbf b$.  Then
\begin{align*}
\mathbf c\cdot \mathbf n 
&= \lambda \mathbf a\cdot\mathbf n+(1-\lambda)\mathbf b\cdot \mathbf n \\
&\leq \lambda(\nu+K_a)+(1-\lambda)(\nu+K_b)\\
&=\nu+ \lambda K_a+(1-\lambda)K_b.
\end{align*}
It follows that $\mathbf c\in L$, with associated constant $\lambda K_a+(1-\lambda)K_b$.  Thus $L$ is convex.

For the second claim, simply note that if $\mathbf a\in P$ with associated constant $K$,
then for every subcubic graph $G$, 
say with parameters  $\mathbf n=(n_3,n_2,n_1)$ and matching number $\nu$, we have
$\mathbf b\cdot\mathbf n\leq \mathbf a\cdot\mathbf n \leq \nu +K$, so $\mathbf b\in L$ with associated constant $K$.

Now for the final part.
Let $K$ be such that $\nu(G)\geq x_3n_3+x_2n_2+x_1n_1-K$ for every
connected subcubic graph $G$. By Lemma~\ref{n1n3} we have $n_3\geq n_1-2$, and so 
$(x_3-\lambda)n_3+x_2n_2+(x_1+\lambda)n_1-(K+2\lambda)\leq
x_3n_3+x_2n_2+x_1n_1-K\leq\nu(G)$, which shows that
$(x_3-\lambda,x_2,x_1+\lambda)\in L$.
\end{proof}

The next lemma will allow us to restrict our attention to $P_+$.

\begin{lemma}\label{Pplus}
If $P_+\subseteq L$ then $P\subseteq L$.
\end{lemma}

\begin{proof}
Consider $x=(x_3,x_2,x_1)\in P\setminus L$. Our aim is to find a point in
$P_+\setminus L$. If each $x_i$ is non-negative then $x$ is such a
point, so we assume the contrary.

First suppose $x_2<0$. We claim that $x'=(x_3,0,x_1)\in P$. Since $x\in
P$, the first and third inequalities defining $P$ are immediate for
$x'$, and the second is trivial. The fourth and sixth inequalities
follow from the first, and the fifth follows from the third. Therefore
$x'\in P$. Now if $x'\in L$ then $x\in L$ because $L$ is downward
closed, contradicting our choice of $x$. Thus $x'\in P\setminus L$.

Therefore we may assume that $x_2\geq 0$. Next we consider the case in
which $x_3<0$. Set $\lambda=-x_3$ and let
$x'=(x_3+\lambda,x_2,x_1-\lambda)=(0,x_2,x_1+x_3)$. We claim that
$x'\in P$. The first inequality for $P$ is trivial, and the second,
third and fifth are true because $x\in P$. The fourth and sixth
inequalities are implied by the second. Thus $x'\in P$. If $x'\in L$
then by Lemma \ref{lamb} the point
$(x_3+\lambda-\lambda,x_2,x_1-\lambda+\lambda)=x\in L$, contradicting
our choice of $x$. Therefore $x'\in P\setminus L$ and we may assume
$x_3\geq 0$.

Finally suppose $x_1<0$. Then we claim $x'=(x_3,x_2,0)\in P\setminus
L$. To check $x'\in P$ observe that the first, second, fourth and
sixth inequalities are true because $x\in P$. The third follows from
the first and the fifth follows from the first and second. Again we may conclude
$x'\notin L$ because $L$ is downward closed. Hence $x'\in P\setminus
L$ as required, completing the proof that $P_+\subseteq L$ implies
$P\subseteq L$.
\end{proof}

It is therefore enough to prove that $P_+\subseteq L$.
Since $L$ is a convex set, it is enough to show that the extreme points of
$P_+$ all belong to $L$. 
The extreme points of $P_+$ (written as $(x_3,x_2,x_1)$)  are
\begin{align*}
\{(0,1/2,1/2),
(0,1/3,2/3), 
(1/4,1/2,1/4), 
(7/16,3/8,3/16), \\
(4/9,1/3,2/9),
(1/4,1/2,0), 
(7/16,3/8,0),
(0,1/2,0), 
(4/9,0,0),\\
(0,0,0),
(4/9,1/3,0),
(0,0,2/3), 
(4/9,0,2/9)\}.
\end{align*}
This can be verified by hand, or (as we did) by using a computational package such as {\em polymake} \cite{polymake}.

Our aim is then to show that all thirteen extreme points of $P_+$ belong to $L$.
Since $L$ is downward closed, it is enough to consider the points that do not lie below any others: for instance, $(7/16,3/8,0)$
lies below $(7/16,3/8,3/16)$, so $(7/16,3/8,3/16)\in L$ implies that $(7/16,3/8,0)\in L$.  This leaves us with the following five points:
\begin{align*}
\{(0,1/2,1/2),
(0,1/3,2/3), 
(1/4,1/2,1/4), 
(7/16,3/8,3/16),
(4/9,1/3,2/9)\}.
\end{align*}
The fact that these points all belong to $L$ follows from Theorem \ref{main}, which we prove in the next section.  We conclude that $P\subseteq L$.

\section{Proofs of Theorem~\ref{main} and Corollary~\ref{pmcor}}\label{pmainb}

First we show how Corollary \ref{pmcor} follows from 
Theorem \ref{main}.

\begin{proof}
Let $G$ be a connected subcubic graph. 
Observe that by Theorem~\ref{main} and monotonicity, we have
$\nu(G)\geq\alpha n_3+\beta n_2+\gamma n_1-1$  for each extreme point
$(\alpha,\beta,\gamma)$ of $P_+$. By convexity, the same inequality
holds for every point $(\alpha,\beta,\gamma)\in P_+$. 

Now suppose $(\alpha,\beta,\gamma)\in P$ and $\alpha\geq 0$. Then
(arguing as in the proof of Lemma~\ref{Pplus}) we know that
$(\alpha,\beta',\gamma')\in P_+$ where $\beta'=\max\{\beta,0\}$ and
$\gamma'=\max\{\gamma,0\}$. Hence
$$\nu(G)\geq\alpha n_3+\beta' n_2+\gamma' n_1-1\geq\alpha n_3+\beta
n_2+\gamma n_1-1.$$ 

If $\alpha<0$ then set $\lambda=|\alpha|$. Then as in the proof of Lemma~\ref{Pplus} we find
that
$(\alpha+\lambda,\beta,\gamma-\lambda)=(0,\beta,\gamma-\lambda)\in
P$. Hence by the previous paragraph  
$\nu(G)\geq\beta n_2+(\gamma-\lambda)n_1-1$. By Lemma~\ref{n1n3} we
have $2\lambda\geq\lambda n_1-\lambda n_3$. Summing these two
inequalities and rearranging gives $\nu(G)\geq\alpha n_3+\beta
n_2+\gamma n_1-(2\lambda+1)$ as required. 
\end{proof}

The remainder of this section is devoted to the proof of Theorem \ref{main}.  

\begin{lemma}\label{pm} Let $G$ be a connected subcubic graph with $n$
  vertices. Suppose $\nu(G)\geq (n-1)/2$. Then $G$ satisfies Theorem~\ref{main}.
\end{lemma}

\begin{proof}
Bounds \eqref{b1} and \eqref{b3} are immediate. Bound \eqref{b4} holds unless
$7n/16-1/8>n/2-1/2$, which implies $n\leq 5$. If \eqref{b5} fails to hold
then $4n/9-1/9>n/2-1/2$, which means $n\leq 6$. These cases are easily
checked. For \eqref{b2}, using Lemma~\ref{n1n3} we find $n_1\leq n_3+2\leq n-n_1+2$, and hence $n_1\leq 1+n/2$. Thus
$n_2/3+2n_1/3-1\leq n/3+n_1/3-1\leq n/2+1/3-1$.
\end{proof}

In particular, if $G$ has a perfect matching or if $G$ is
hypomatchable (meaning $G-v$ has a perfect matching for every $v\in
V(G)$) then Theorem \ref{main} holds.

In our proof we will make use of the Gallai-Edmonds structure
theorem (see, for instance, \cite{LP}). In the statement below, the sets $A$, $B$ and $C$ are defined
as follows (here $\Gamma(A)$ denotes the neighbourhood of $A$).
\begin{itemize}
\item $A=\{v\in V(G):\nu(G-v)=\nu(G)\}$,
\item $B=\Gamma(A)\setminus A$,
\item $C=V(G)\setminus(A\cup B)$.
\end{itemize}

\begin{theorem}\label{GE}(Gallai-Edmonds) Let $G$ be a graph. Then
\begin{enumerate}
\item every component of $G[A]$ is hypomatchable,
\item every component of $G[C]$ has a perfect matching,
\item every $X\subseteq B$ has neighbours in at least $|X|+1$ components of $G[A]$.  
\end{enumerate}
\end{theorem}


One consequence of Theorem~\ref{GE} is that we may assume
$B\not=\emptyset$, otherwise each component of $G$ has a perfect
  matching or is hypomatchable, in which case we are
  done by Lemma~\ref{pm}. Note also that Part (3) implies that each vertex of $B$ has degree at least two.

It is easy to check that all the bounds in Theorem~\ref{main} hold for
graphs with at most three vertices, so we assume $G$ has $n\geq 4$ 
vertices and that the theorem is true for graphs with fewer than $n$ vertices.
Since we may consider each component separately, we may assume $G$ is connected.
Choose a vertex  $v\in B$, and consider $G-v$. Since $v\notin A$ we                      
know $\nu(G-v)=\nu(G)-1$. Let $t_i$ denote the number of
neighbours of $v$ of degree $i$ for $i=1,2,3$. Let $U$ denote the
set of neighbours of $v$ of degree 1, so $|U|=t_1$.
Then $G'=G-v-U$ satisfies 
$\nu(G')=\nu(G)-1$. 

Let $n_i'$ denote the number of vertices of degree $i$ in $G'$. Since each degree-3 neighbour of $v$ becomes a
degree-2 vertex, the number of degree-3 vertices drops by $t_3$, plus
one more if $v$ itself has degree 3. Thus 
$n_3'=n_3-t_3-(d(v)-2)=n_3-t_3-(t_1+t_2+t_3-2)=n_3-2t_3-t_2-t_1+2$. Each
degree-2 neighbour of $v$ becomes a degree-1 vertex, and if $v$ has
degree 2 then the number of degree-2 vertices drops by one more. Hence
$n_2'=n_2+t_3-t_2-(3-d(v))=n_2+t_3-t_2-(3-t_1-t_2-t_3))=n_2+2t_3+t_1-3$. Finally
$n_1'=n_1-t_1+t_2$, and $c'\leq 
t_3+t_2$. Then by the induction hypothesis,

\begin{enumerate}
\item $
\begin{aligned}[t]
\nu(G')&\geq n_2'/2+n_1'/2-c'/2\\
       &\geq
  n_2/2+(2t_3+t_1-3)/2+n_1/2+(t_2-t_1)/2-(t_3+t_2)/2\\
       &=n_2/2+n_1/2-1/2+(t_3-2)/2,
\end{aligned}
$
\item $
\begin{aligned}[t]
\nu(G')&\geq n_2'/3+2n_1'/3-c'\\
       &\geq
  n_2/3+(2t_3+t_1-3)/3+2n_1/3+2(t_2-t_1)/3-(t_3+t_2)\\
       &=n_2/3+2n_1/3-1-(t_3+t_2+t_1)/3,
\end{aligned}
$
\item $
\begin{aligned}[t]
\nu(G')&\geq n_3'/4+n_2'/2+n_1'/4-c'/2\\
       &\geq n_3/4+(2-2t_3-t_2-t_1)/4+n_2/2+(2t_3+t_1-3)/2+n_1/4\\
&\qquad +(t_2-t_1)/4-(t_3+t_2)/2\\
       &=n_3/4+n_2/2+n_1/4-1/2-(t_2+1)/2,
\end{aligned}
$
\item $
\begin{aligned}[t]
\nu(G')&\geq 7n_3'/16+3n_2'/8+3n_1'/16-c'/8\\
&\geq 7n_3/16+7(2-2t_3-t_2-t_1)/16+3n_2/8+3(2t_3+t_1-3)/8\\
&\qquad +3n_1/16+3(t_2-t_1)/16-(t_3+t_2)/8\\
       &= 7n_3/16+3n_2/8+3n_1/16-1/4-t_3/4-3t_2/8-t_1/4\\
       &=[7n_3/16+3n_2/8+3n_1/16-1/8]-(4t_3+6t_2+4t_1+2)/16,
\end{aligned} 
$
\item $
\begin{aligned}[t]
\nu(G')&\geq 4n_3'/9+n_2'/3+2n_1'/9-c'/9\\
       &\geq 4n_3/9+4(2-2t_3-t_2-t_1)/9+n_2/3+(2t_3+t_1-3)/3\\
&\qquad +2n_1/9+2(t_2-t_1)/9-(t_3+t_2)/9\\
       &=4n_3/9+n_2/3+2n_1/9-1/9-(t_3+t_2+t_1)/3.
\end{aligned}
$
\end{enumerate}
Since $\nu(G)=\nu(G')+1$ and $t_3+t_2+t_1\leq 3$ it follows from the calculations above that bounds \eqref{b1}, \eqref{b2} and \eqref{b5} hold
for $G$. (In fact \eqref{b2} alternatively follows from \eqref{b5} together with Lemma~\ref{lamb}(3)).

We now focus on bounds \eqref{b3} and \eqref{b4}. Note that in these cases, our inductive statement gives
$$\nu(G')\geq n_3/4+n_2/2+n_1/4-1/2-(t_2+1)/2,$$
and
$$\nu(G')\geq [7n_3/16+3n_2/8+3n_1/16-1/8]-(4t_3+6t_2+4t_1+2)/16.$$

First we note some consequences of
Theorem~\ref{GE} and the above calculations.

\begin{lemma}\label{cons}
\begin{enumerate}
\item Every $v\in B$ has at least two neighbours in $A$.
\item If $x\in A$ has exactly two neighbours $u$ and $w$, and if $u\in
  B$, then $w\in B$ as well.
\item If \eqref{b4} fails for $G$ then every $v\in B$ has degree 3.
\item If one of \eqref{b3} and \eqref{b4} fails for $G$ then every $v\in B$ has at
  least two degree-2 neighbours.
\end{enumerate}
\end{lemma}

\begin{proof}
We have already noted that the first statement is immediate from Theorem~\ref{GE}(3). To verify
the second claim, observe that if $w\in A$ then $u$ and $w$ are both
in a component $H$ of $G[A]$, which is hypomatchable by
Theorem~\ref{GE}. But $x$ has degree 1 in $H$, which is not possible
in a hypomatchable component. Thus $w\in B$.

If \eqref{b3} fails then $t_2\ge2$; if \eqref{b4} fails then 
$4t_3+6t_2+4t_1\ge15$ and so (as $d(v)\le3$) we have $t_2\ge2$ and $t_1+t_2+t_3=3$.
The last two assertions follow immediately, as the same calculation holds for any vertex of $B$.
\end{proof}

Next we derive some elementary facts about the neighbours of degree-2 vertices.

\begin{lemma}\label{twotwo}
Suppose $G$ fails to satisfy one of \eqref{b3} and \eqref{b4}. Then no two degree-2 vertices of $G$ are adjacent. Furthermore every vertex of $B$ has degree 3.
\end{lemma}

\begin{proof}
Recall our assumption that $G$ has at least four vertices. If $G$ is a 4-cycle then
\eqref{b3} and \eqref{b4} are satisfied (by Lemma~\ref{pm}), so
let us assume otherwise. Suppose $u$ and $w$ are adjacent
degree-2 vertices.

If $u$ and $w$ are not in a triangle or 4-cycle then
suppressing $u$ and $w$ 
(i.e.~if $u'$ and $v'$ are the other neighbours of $u$, $v$ then we replace the path $u'uvv'$ by the edge $u'v'$)
gives a connected graph $G'$ with
$\nu(G')=\nu(G)-1$,
$n_3'=n_3$, $n_2'=n_2-2$, and $n_1'=n_1$. Then by the
induction hypothesis for \eqref{b3},
$\nu(G')\geq n_3'/4+n_2'/2+n_1'/4-1/2=                                 
  n_3/4+n_2/2+n_1/4-1/2-1$,
showing $G$ satisfies \eqref{b3}. For \eqref{b4} we have by induction
$\nu(G')\geq 7n_3'/16+3n_2'/8+3n_1'/16-1/8=
  7n_3/16+3n_2/8+3n_1/16-1/8-6/8$,
which also suffices.

If $uwx$ is a triangle then form $G'$ by removing $u$ and $w$. Then
$\nu(G')=\nu(G)-1$,
$n_3'=n_3-1$,
$n_2'=n_2-2$, $n_1'=n_1+1$, and $c'=1$.
For \eqref{b3} we get 
$\nu(G')\geq n_3'/4+n_2'/2+n_1'/4-1/2=
  n_3/4-1/4+n_2/2-1+n_1/4+1/4-1/2=[n_3/4+n_2/2+n_1/4-1/2]-1$,
showing $G$ satisfies \eqref{b3}. For \eqref{b4} we have by induction
$\nu(G')\geq  7n_3/16-7/16+3n_2/8-6/8+3n_1/16+3/16-1/8=[7n_3/16+3n_2/8+3n_1/16-1/8]-1$,
as needed.

If $u$ and $w$ are in a 4-cycle $uwxz$ then by assumption (say)
$x$ has degree 3. Form $G'$ by removing $u$ and $w$, so that $\nu(G')=\nu(G)-1$.
If $d(z)=3$ then $G'$ has $n_3'=n_3-2$, $n_2'=n_2$, $n_1'=n_1$, and $c'=1$.
Then using induction for \eqref{b3} we find
$\nu(G')\geq n_3'/4+n_2'/2+n_1'/4-1/2= 
  (n_3-2)/4+n_2/2+n_1/4-1/2=n_3/4+n_2/2+n_1/4-1/2-1/2$,
which suffices. For \eqref{b4} we get
$\nu(G')\geq  7n_3/16-14/16+3n_2/8+3n_1/16-1/8=[7n_3/16+3n_2/8+3n_1/16-1/8]-14/16$ as required.

If $d(z)=2$ the parameters become $n_3'=n_3-1$,
$n_2'=n_2-2$, and $n_1'=n_1+1$, giving for \eqref{b3}
$\nu(G')\geq n_3'/4+n_2'/2+n_1'/4-1/2=
  (n_3-1)/4+(n_2-2)/2+n_1/4=1/4-1/2+n_3/4+n_2/2+n_1/4-1/2-1$ as needed.
For \eqref{b4} we get
$\nu(G')\geq 
7n_3/16-7/16+3n_2/8-6/8+3n_1/16+3/16-1/8=[7n_3/16+3n_2\
/8+3n_1/16-1/8]-1$.
This completes the proof of the first statement. The second statement now follows using Lemma~\ref{cons}(3),(4).
\end{proof}

\begin{lemma}\label{deg23} Suppose $G$ fails to satisfy one of \eqref{b3} and
  \eqref{b4}. Then each degree-2 vertex $w$ has two degree-3 neighbours.
\end{lemma}

\begin{proof} Lemma~\ref{twotwo} tells us that $w$ has no degree-2 neighbours. Suppose for a contradiction that $w$ has a
degree-1 neighbour $x$. Then (recalling $G$ has at least four vertices) $G'-\{w,x\}$ has $\nu(G')=\nu(G)-1$,
$n_3'=n_3-1$,
$n_2'=n_2$, $n_1'=n_1-1$, and $c'=1$. Then using induction for \eqref{b3} gives $\nu(G')\geq n_3'/4+n_2'/2+n_1'/4-c'/2\geq                                                                
n_3/4-1/4+n_2/2+n_1/4-1/4-1/2= [n_3/4+n_2/2+n_1/4-1/2]-1/2$, which suffices. For \eqref{b4} we get $\nu(G')\geq
7n_3'/16+3n_2'/8+3n_1'/16-c'/8\geq                      
7n_3/16-7/16+3n_2/8+3n_1/16-3/16-1/8=[7n_3/16+3n_2/8+3n_1/16-1/8]-10/16$. 
\end{proof}

Call a degree-3 vertex $v\in G$ {\it good} if it has two degree-2 neighbours that
do not have a common neighbour different from $v$.
Observe that if $v$ has three degree-2 neighbours then either $v$ is
good, or $G=K_{2,3}$, in which case \eqref{b3} and \eqref{b4} hold.

\begin{lemma}\label{good} Suppose $G$ fails to satisfy one of \eqref{b3} and
  \eqref{b4}. Then every good vertex $v$ of $G$ has three degree-2 neighbours,
  all of which are in different components of $G-v$.
\end{lemma} 

\begin{proof}
Let $w$ and $x$ be degree-2 neighbours that are not adjacent and
have no common neighbour other than $v$. As before, we write $t_i$ for the number of degree $i$ neighbours of $v$,
and $U$ for the set of degree 1 neighbours of $v$. Let $G'$ be the graph
obtained by removing $\{v\}\cup U$ and identifying $w$ and $x$ into a new vertex
of degree 2. Then $\nu(G')=\nu(G)-1$, $n_3'=n_3-t_3-1$,
$n_2'=n_2-t_2+t_3+1$, $n_1'=n_1-t_1+t_2-2$, and $c'\leq 2-t_1$.

The computation for \eqref{b3} becomes 
$\nu(G')\geq n_3'/4+n_2'/2+n_1'/4-c'/2\geq
n_3/4-t_3/4-1/4+n_2/2+(t_3+1-t_2)/2+n_1/4+(t_2-t_1-2)/4-(2-t_1)/2=
[n_3/4+n_2/2+n_1/4-1/2]+t_3/4-t_2/4+t_1/4-3/4$. Then \eqref{b3} holds unless $t_2=3$ and $c'=2$.

For \eqref{b4} we get  
$\nu(G')\geq 7n_3'/16+3n_2'/8+3n_1'/16-c'/8\geq
7n_3/16-7t_3/16-7/16+3n_2/8+3(t_3+1-t_2)/8+3n_1/16+3(t_2-t_1-2)/16-(2-t_1)/8
=[7n_3/16+3n_2/8+3n_1/16-1/8]-(t_3+3t_2+t_1+9)/16$, so \eqref{b4} holds unless $t_2=3$ and $c'=2$.

Hence in both cases we may assume that $t_2=3$ and so $c'=2$. Let $y$ be the third neighbour of
$v$. Since $c'=2$ we know that $y$ is 
in a different component of $G'$ (and hence of $G-v$) to $w$ and $x$.  In particular, $y$ is
not adjacent to $w$ or $x$ and
does not share a second common neighbour with either of them. Thus we
could apply the above argument with $w$ and $y$ and find that $x$ is
in a different component of $G-v$ from both $w$ and $y$. This
completes the proof.
\end{proof}

We may now complete the proof for \eqref{b3}.

\begin{lemma}\label{bound3}
$G$ satisfies \eqref{b3}.
\end{lemma}

\begin{proof}
Suppose the contrary. If any
degree-3 vertex has another degree-3 vertex in its
neighbourhood, then we may verify \eqref{b3} by considering the graph
$G'$
obtained by deleting an edge joining two degree-3 vertices. In this
case $n_3'=n_3-2$, $n_2'=n_2+2$, $n_1'=n_1$ and $c'\leq 2$. Hence
using induction we get $\nu(G)\geq\nu(G')\geq
n_3'/4+n_2'/2+n_1'/4-2/2=n_3/4+n_2/2+n_1/4-1/2$, proving \eqref{b3} as
required.

Thus we may assume no two degree-3 vertices are adjacent. Next we
check that no degree-3 vertex has two degree-1 neighbours. If on the
contrary $x$ has degree-1 neighbours $v$ and $w$, and a third
neighbour $z$ (which necessarily has degree 2,
or else $G$ is $K_{1,3}$ and satisfies \eqref{b3}), form $G'$ by removing
$v$, $w$, and $x$. Then $n_3'=n_3-1$, $n_2'=n_2-1$, $n_1'=n_1-1$,
$c'=1$ and $\nu(G)=\nu(G')+1$. Therefore by induction $\nu(G)\geq
n_3'/4+n_2'/2+n_1'/4-c'/2+1\geq[n_3/4+n_2/2+n_1/4-1/2]-1/4-1/2-1/4+1$,
showing \eqref{b3} holds. Thus every degree-3 vertex has at least two
degree-2 neighbours.
 
Suppose a degree-2 vertex $w$ has neighbours $v$ and $z$ (which both have degree 3 by Lemma~\ref{deg23}).
If $v$ is good then $z$ is also good, since otherwise every other
degree-2 neighbour of $z$ (at least one of which exists) is also a
degree-2 neighbour of $v$, and would therefore be in the same
component of $G-v$ as $w$, contradicting Lemma~\ref{good}.
Therefore there are
no good vertices at all, since otherwise (since $G$ has at least one
degree-3 vertex, in $B$) by Lemma \ref{good} we would find that $G$ is a
subdivision of a connected 3-regular graph, but removing any 
degree-3 vertex results in 3 components. This is not possible since,
in particular, every connected graph has a vertex whose removal leaves
a connected graph. 

Since $G$ has no good vertices, in particular no
 degree-3 vertex can
have three degree-2
neighbours. 
So every degree 3 vertex has exactly two degree 2 neighbours.  It follows that 
$G$ is a cycle (of even length) with a pendant edge attached to
every second vertex (these are the graphs $G_3(t)$ in Example 3 in the next section). But \eqref{b3} holds for this graph, completing the
proof. 
\end{proof}

We are left to verify \eqref{b4}. We need one more technical lemma.

\begin{lemma}
No vertex in $B$ is good.
\end{lemma}

\begin{proof}
Suppose on the contrary that $B$ contains good vertices. Let $v\in B$
be a good vertex. Let $W$ be the union of the vertex sets of all paths of
the form $vw_1w_2\ldots w_r$ where $r\geq 1$, each $w_i$ with $i$
odd is a degree-2 vertex in $A$, and each $w_i$ with $i$ even is in
$B$. Let $H$ be the subgraph of $G$ induced by $W$. Then $H$ is
connected.

We claim that each vertex of $W\cap B$ is good. To verify this,
consider a good vertex $w\in W\cap B$ (for example $w=v$). By
Lemma~\ref{cons}(1) we know $w$ has at least two neighbours $u$ and $x$
in $A$, and $d(u)=d(x)=2$ by Lemma~\ref{good}. Also, Lemma~\ref{cons}(2) implies that the other neighbour $z$ of $u$ is in
$B$ and hence is in $W\cap B$. Thus $d(z)=3$ by Lemma~\ref{cons}(3). If $z$ were not good then every degree-2
neighbour of $z$ different from $u$ (at least one of which exists, 
by Lemma~\ref{cons}(4)) would be a degree-2 neighbour of $w$, and
would hence be in the same component of $G-w$ as $u$, contradicting
Lemma~\ref{good}. Hence $z$ is good. Applying this observation repeatedly (moving along the paths used to define $H$) we
find that every vertex of $W\cap B$ is good.

By Lemma~\ref{cons}(2) we know that $A\cap W$ is independent, and each
$x\in A\cap W$ has exactly two neighbours in $B\cap W$. Since each
$w\in B\cap W$ is good, it has three degree-2 neighbours in $G$ by
Lemma~\ref{good}, at least two of which are in $A$ by
Lemma~\ref{cons}(1). So by Lemma~\ref{cons}(3) we know $B\cap W$ is independent.
Therefore $H$ is the subdivision of a connected subcubic graph $J$ with vertex set
$B\cap W$ and minimum degree at least 2. (Note that $J$ has no multiple edges by Lemma~\ref{good} and the fact that each $w\in B\cap W$ is good.)

Since each $w\in B\cap W$ is good, the graph $J$
has the property 
that $J-y$ has $d(y)\geq 2$ components for every vertex $y$ of
$J$. Such a graph cannot exist, so the proof is complete.
\end{proof}

We may therefore assume that no vertex in
$B$ has three degree-2 neighbours. 
Choose $v\in B$. By Lemma~\ref{twotwo} we have $d(v)=3$, and by Lemma~\ref{cons}(4) we know that
$v$ has at least two degree-2 neighbours, say $w$ and $x$. By
Lemma~\ref{cons}(1) at least one of them, say $w$, is in $A$. Since
$v$ is not good, the other neighbour $z$ of $v$ is not a degree-2
vertex, and $w$ and $x$ have another common neighbour $y$.
By Lemma~\ref{cons}(2) we know
$y$ is in $B$. Then by Lemma~\ref{cons}(3) we have that $y$ has
another neighbour $u$, and $d(u)\not=2$ since $y$ is not good.  
Since \eqref{b4} holds for $K_4$
with one edge deleted, we may assume $u\not= v$. 
If $G$ consists of a 4-cycle plus two
pendant edges attached to non-adjacent vertices then \eqref{b4} holds, so we
may assume without loss of generality that $z$ has degree 3. 

If $z=u$ remove $v,w,x,y$. Then $n_3'=n_3-3$, $n_2'=n_2-2$,
$n_1'=n_1+1$, $c'=1$ and $\nu(G')=\nu(G)-2$. Then by induction
$\nu(G)\geq\nu(G')+2\geq  7n_3/16+3n_2/8+3n_1/16-1/8-30/16+2$, which implies our result.

If $u$ has degree 1 we remove $u,v,w,x,y$. Then $n_3'=n_3-3$,
$n_2'=n_2-1$,
$n_1'=n_1-1$, $c'=1$ and $\nu(G')=\nu(G)-2$. Then by induction
$\nu(G)\geq\nu(G')+2\geq  7n_3/16+3n_2/8+3n_1/16-1/8-30/16+2$, as needed.

Otherwise $z\not= u$, and $d(z)=d(u)=3$. In this case we remove $v,w,x,y$. Then
$n_3'=n_3-4$, $n_2'=n_2$,
$n_1'=n_1$, $c'\leq 2$ and $\nu(G')=\nu(G)-2$. Then by induction
$\nu(G)\geq\nu(G')+2\geq  7n_3/16+3n_2/8+3n_1/16-1/8-30/16+2$, which
completes the proof of Theorem \ref{main}.

\section{$L\subseteq P$}\label{pmain2}

The fact that $L\subseteq P$ is an immediate consequence of Theorem~\ref{main2}, which we prove in this section.

Suppose that $(x_3,x_2,x_1)\in L$, so there is some real number $K$ such that
\begin{equation}\label{test}
\nu(G)\ge x_3n_3(G)+x_2n_2(G)+x_1n_1(G)-K
\end{equation}
for every connected subcubic graph $G$ (where $n_i(G)$ denotes the number of vertices of $G$ of degree $i$).  We fix a choice of $(x_3,x_2,x_1)$ and $K$ for the rest of this section.  

We will consider six special families of graphs: each family will show that $(x_3,x_2,x_1)$ 
must satisfy one of the inequalities in the definition of $P$. An example from each family is shown in the figures.

\medskip\noindent{\bf Example 1.}
Let $t$ be an odd positive integer. The graph $G_1(t)$ is the tree
with a root plus $t+1$ levels, indexed 
by $i=0,\ldots,t$, in which level $i$ contains $3\cdot 2^{i}$ vertices,
and all vertices except the leaves have degree 3.  Thus $G_1(t)$ is (internally) a cubic tree and has depth $t+1$.
Then $n_1=3\cdot 2^t$,
$n_2=0$ and $n_3=1+3(2^t-1)=3\cdot 2^t-2$. Since $G_1(t)$ is bipartite
with one partition class $S$ formed by the vertices at levels
$0,2,\ldots,t-1$ we see $\nu(G_1(t))\leq|S|=3(4^{(t+1)/2}-1)/3=2^{t+1}-1$.
By \eqref{test} we must have
$$(3\cdot 2^t-2)x_3+3\cdot 2^tx_1-K\leq 2\cdot 2^t-1,$$
and so, dividing by $3\cdot 2^t$ and letting $t\to\infty$, we see that
$$x_3+x_1\leq 2/3.$$

\begin{figure}

\begin{tikzpicture}
  [scale=.6,auto=left,
vx/.style={circle,draw,minimum size=4pt,inner sep=0pt},
capt/.style={rectangle,minimum size=4pt,inner sep=0pt}]

\node at (0,0) (base) {};
\node (n1) at ($(base)+(5,1)$) [vx] {};
\node (n2) at ($(base)+(2,3)$) [vx] {};
\node (n3) at ($(base)+(5,3)$) [vx] {};
\node (n4) at ($(base)+(8,3)$) [vx] {};
\node (n5) at ($(base)+(1,5)$) [vx] {};
\node (n6) at ($(base)+(3,5)$) [vx] {};
\node (n7) at ($(base)+(4,5)$) [vx] {};
\node (n8) at ($(base)+(6,5)$) [vx] {};
\node (n9) at ($(base)+(7,5)$) [vx] {};
\node (n10) at ($(base)+(9,5)$) [vx] {};

\node at ($(base)+(5,-1)$) [capt] {$G_1(1)$}; 

\foreach \from/\to in {n1/n2,n1/n3,n1/n4,n2/n5,n2/n6,n3/n7,n3/n8,n4/n9,n4/n10}
\draw (\from) -- (\to);

\node at (10,0) (base) {};
\node (n1) at ($(base)+(5,1)$) [vx] {};
\node (n2) at ($(base)+(2,3)$) [vx] {};
\node (n3) at ($(base)+(5,3)$) [vx] {};
\node (n4) at ($(base)+(8,3)$) [vx] {};
\node (n5) at ($(base)+(1,4.5)$) [vx] {};
\node (n6) at ($(base)+(2.5,4.5)$) [vx] {};
\node (n7) at ($(base)+(4.25,4.5)$) [vx] {};
\node (n8) at ($(base)+(5.75,4.5)$) [vx] {};
\node (n9) at ($(base)+(7.5,4.5)$) [vx] {};
\node (n10) at ($(base)+(9,4.5)$) [vx] {};

\node at ($(base)+(5,-1)$) [capt] {$G_2(1)$}; 

\foreach \from/\to in {n1/n2,n1/n3,n1/n4,n2/n5,n2/n6,n3/n7,n3/n8,n4/n9,n4/n10}
\draw (\from) -- (\to);

\foreach \x in {n5,n6,n7,n8,n9,n10}
{
\node at (\x) (newbase) [vx] {}; 
\node (q1) at ($(newbase)+(-0.5,0.5)$) [vx] {};
\node (q2) at ($(newbase)+(-0.5,1.5)$) [vx] {};
\node (q4) at ($(newbase)+(0.5,0.5)$) [vx] {};
\node (q3) at ($(newbase)+(0.5,1.5)$) [vx] {};
\foreach \from/\to in {newbase/q1,q1/q2,q2/q3,q3/q4,q4/newbase,q1/q3,q2/q4}
\draw (\from) -- (\to);
};

\end{tikzpicture}

\end{figure}

\medskip\noindent{\bf Example 2.}
Let $J$ denote the graph obtained by subdividing one edge of $K_4$,
and let $x$ denote the single vertex of degree 2 in $J$.
We define the graph $G_2(t)$, again for odd $t$, by identifying each
leaf in $G_1(t)$ with the vertex $x$ in a copy of the graph $J$, such
that all copies are disjoint from each other and the rest of the
graph. Then for this graph $n_1=n_2=0$, and
$n_3=3\cdot 2^t-2+15\cdot 2^t=9\cdot 2^{t+1}-2$. The same set $S$ as before now has the
property that removing it leaves
$1+3(2+2^3+\ldots+2^t)=1+6(4^{(t+1)/2}-1)/3=2^{t+2}-1$ odd
  components. Therefore any maximum matching in $G$ must leave exposed
  at least $2^{t+2}-1-|S|=2^{t+1}$ vertices. This tells us
  $\nu(G_2(t))\leq(9\cdot 2^{t+1}-2-2^{t+1})/2=(2^{t+4}-2)/2=2^{t+3}-1$.
So \eqref{test} implies that
$$(9\cdot 2^{t+1}-2)x_3-K\leq 2^{t+3}-1.$$
Dividing by $9\cdot 2^{t+1}$ and taking a limit gives
$$x_3\leq 4/9.$$

\medskip\noindent{\bf Example 3.}
Let $t\ge2$ be a positive integer.  The graph $G_3(t)$ is obtained from the
cycle with $2t$ vertices by attaching a pendant edge to every second
vertex. Then $n_1=n_2=n_3=t$. The graph is bipartite
with one vertex class consisting of the vertices of degree 3,
so $\nu(G_3(t))\leq n_3=t$.  Thus
$$x_3t+x_2t+x_1t-K\leq t.$$
Dividing by $t$ and taking a limit gives
$$x_3+x_2+x_1\leq1.$$

\medskip\noindent{\bf Example 4.}
The graph $G_4(t)$ is obtained from $G_3(t)$ by adding $t$ disjoint
copies of $J$, identifying the vertex $x$ in each copy with the leaf
of a pendant edge. Then $n_1=0$, $n_2=t$ and $n_3=6t$. The set of
degree-3 vertices on the cycle has size $t$ and leaves $2t$ odd
components when deleted, showing $\nu(G_4(t))\leq (7t-t)/2=3t$.
Thus
$$6x_3t+x_2t-K\leq 3t.$$
Dividing by $6t$ and taking a limit gives
$$x_3+x_2/6\leq 1/2.$$

\begin{figure}

\begin{tikzpicture}
  [scale=.6,auto=left,
vx/.style={circle,draw,minimum size=4pt,inner sep=0pt},
capt/.style={rectangle,minimum size=4pt,inner sep=0pt}]

\node at (0,0) (base) {};
\node (n1) at ($(base)$) [vx] {};
\node (n2) at ($(base)+(1,1)$) [vx] {};
\node (n3) at ($(base)+(0,2)$) [vx] {};
\node (n4) at ($(base)+(-1,1)$) [vx] {};
\node (n5) at ($(base)+(2,1)$) [vx] {};
\node (n6) at ($(base)+(-2,1)$) [vx] {};

\node at ($(base)+(0,-1)$) [capt] {$G_3(2)$}; 

\foreach \from/\to in {n1/n2,n2/n3,n3/n4,n4/n1,n2/n5,n4/n6}
\draw (\from) -- (\to);

\node at (7.75,0) (base) {};
\node (n1) at ($(base)$) [vx] {};
\node (n2) at ($(base)+(1,1)$) [vx] {};
\node (n3) at ($(base)+(0,2)$) [vx] {};
\node (n4) at ($(base)+(-1,1)$) [vx] {};
\node (n5) at ($(base)+(1.75,1)$) [vx] {};
\node (n6) at ($(base)+(-1.75,1)$) [vx] {};

\node at ($(base)+(0,-1)$) [capt] {$G_4(2)$}; 

\foreach \from/\to in {n1/n2,n2/n3,n3/n4,n4/n1,n2/n5,n4/n6}
\draw (\from) -- (\to);

\foreach \x/\y in {n5/1,n6/-1}
{
\node at (\x) (newbase) [vx] {}; 
\node (q1) at ($(newbase)+\y*(0.5,-0.5)$) [vx] {};
\node (q2) at ($(newbase)+\y*(1.5,-0.5)$) [vx] {};
\node (q4) at ($(newbase)+\y*(0.5,0.5)$) [vx] {};
\node (q3) at ($(newbase)+\y*(1.5,0.5)$) [vx] {};
\foreach \from/\to in {newbase/q1,q1/q2,q2/q3,q3/q4,q4/newbase,q1/q3,q2/q4}
\draw (\from) -- (\to);
};

\node at (13.5,0) (base) {};
\node (n1) at ($(base)$) [vx] {};
\node (n12) at ($(base)+(1,0)$) [vx] {};
\node (n2) at ($(base)+(2,0)$) [vx] {};
\node (n23) at ($(base)+(2,1)$) [vx] {};
\node (n3) at ($(base)+(2,2)$) [vx] {};
\node (n34) at ($(base)+(1,2)$) [vx] {};
\node (n4) at ($(base)+(0,2)$) [vx] {};
\node (n41) at ($(base)+(0,1)$) [vx] {};
\node (n5) at ($(base)+(0.75,1)$) [vx] {};
\node (n6) at ($(base)+(1.25,1)$) [vx] {};

\node at ($(base)+(1,-1)$) [capt] {$G_5(4)$}; 

\foreach \from/\to in {n1/n12,n12/n2,n2/n23,n23/n3,n3/n34,n34/n4,n4/n41,n41/n1,n1/n5,n5/n3,n2/n6,n4/n6}
\draw (\from) -- (\to);

\node at (18,0) (base) {};
\node (n1) at ($(base)$) [vx] {};
\node (n2) at ($(base)+(2,0)$) [vx] {};
\node (n3) at ($(base)+(1,1)$) [vx] {};

\node at ($(base)+(1,-1)$) [capt] {$G_6(3)$}; 

\foreach \from/\to in {n1/n2,n2/n3,n3/n1}
\draw (\from) -- (\to);

\end{tikzpicture}

\end{figure}

\medskip\noindent{\bf Example 5.} For each even integer $t\ge4$, let
$G_5(t)$ be obtained from a cubic graph $H$ on $t$ vertices by 
subdividing every edge of $H$ exactly once (for sake of definiteness, we may 
may take $H$ to be a cycle of length $t$ with opposite vertices joined).
Then $n_1=0$, $n_2=e(H)=3t/2$ and $n_3=t$. Then 
$G_5(t)$ is bipartite with one
vertex class $V(H)$ of size $t$, so $\nu(G)\leq t$.  
Thus
$$x_3t+3x_2t/2-K\leq t.$$
Dividing by $t$ and taking a limit gives
$$x_3+3x_2/2\leq 1.$$

\medskip\noindent{\bf Example 6.}  Finally, for odd integers $t\ge3$, we let $G_6(t)$
be the odd cycle of length $t$.  Then
$n_1=n_3=0$ and $n_2=t$, while $\nu=(t-1)/2$.
Thus
$$x_2t/2-K\leq t/2 -1/2.$$
Dividing by $t/2$ and taking a limit gives
$$x_2\leq 1/2.$$

\bigskip

The proof of Theorem \ref{main2} is now immediate.

\begin{proof}[Proof of Theorem \ref{main2}]
If $(x_3,x_2,x_1)\not\in P$ then it fails to satisfy one of the inequalities used to define $P$.
Therefore, taking the example above that corresponds to this inequality (and noting that all the examples are connected) we see that by taking $t$ large we can force $K$ to be arbitrarily large. \end{proof}

In fact it is easy to see that equality holds in each expression
bounding $\nu(G_i(t))$, but we do not need this fact. 
Finally, we note that Example 1 is sharp for \eqref{b2} and \eqref{b5}; Example 2 is sharp for \eqref{b5}; and Example 6 is sharp for \eqref{b1} and \eqref{b3}.

\bigskip
\noindent{\bf Acknowledgements.}  The authors would like to thank G\"unter Ziegler for recommending {\em polymake}, and also an anonymous referee for a very careful reading.

\end{document}